\documentclass{amsart}

\usepackage{amsthm,amssymb,latexsym,amsmath}
\usepackage[all]{xy}
\usepackage[english]{babel}
\usepackage[utf8x]{inputenc}
\usepackage{color}

\newcommand{\CC}{{\mathbb C}}\newcommand{\C}{\mathbb{C}}
\newcommand{\p}[1]{{\mathbb{P}^{#1}}}

\newcommand{\PP}{{\mathbb P}} \newcommand{\bA}{{\mathbb A}}
\newcommand{\op}[1]{{\mathcal O}_{\mathbb{P}^{#1}}}

\newcommand\pt{{\mathrm{\hspace{0.2ex}p\hspace{-0.23ex}t}}}

\newcommand{\OC}{{\mathcal O}_{C}} \newcommand{\IC}{I_{C/\p3}}
\newcommand{\OD}{{\mathcal O}_{D}} \newcommand{\ID}{I_{D/\p3}}

\newcommand{\calc}{{\mathcal C}}
\newcommand{\cald}{{\mathcal D}}
\newcommand{\cale}{{\mathcal E}}

\newcommand{\cali}{{\mathcal I}}

\newcommand{\call}{{\mathcal L}}
\newcommand{\calm}{{\mathcal M}}

\newcommand{\calo}{{\mathcal O}}
\newcommand{\calp}{{\mathcal P}}

\newcommand{\calr}{{\mathcal R}}

\newcommand{\calt}{{\mathcal T}}
\newcommand{\calz}{{\mathcal Z}}

\newcommand{\f}{\varphi}

\newcommand{\EE}{{\mathcal E}}
\newcommand{\HH}{\mathbf H}
\newcommand{\MM}{{\mathcal T}}
\newcommand{\bP}{{\mathcal P}}
\newcommand{\RR}{{\mathcal R}}
\newcommand{\XX}{{\mathcal X}}

\newcommand{\E}{{\mathcal E}}
\newcommand{\F}{{\mathcal F}}
\newcommand{\G}{{\mathcal G}}
\newcommand{\I}{{\mathcal I}}

\newcommand{\Ker}{{\mathcal Ker}}

\newcommand{\Image}{{\mathcal Im}}
\newcommand{\Tor}{{\mathcal Tor}}

\newcommand{\Ext}{\operatorname{Ext}}

\newcommand{\Hom}{\operatorname{Hom}}
\newcommand{\Hilb}{\operatorname{Hilb}}
\newcommand{\h}{\operatorname{h}}
\def\H{\operatorname{H}}
\newcommand{\M}{\operatorname{M}}
\newcommand{\rat}{\operatorname{rat}}
\newcommand{\support}{\operatorname{supp}}

\DeclareMathOperator{\im}{im}

\DeclareMathOperator{\Pic}{{Pic}}

\newcommand{\dual}{{\scriptscriptstyle \operatorname{D}}}
\newcommand{\st}{{\scriptstyle \operatorname{s}}}

\newcommand{\tensor}{\otimes}
\newcommand{\lra}{\longrightarrow}

\newcommand{\ba}{\begin{array}}
\newcommand{\ea}{\end{array}}

\newcommand{\inhom}{{\mathcal H}{\it om}}

\newcommand{\onto}{\twoheadrightarrow}

\newlength{\rrrr}
\newcommand{\isom}[1]{{\settowidth{\rrrr}{$\scriptstyle{x#1x}$}
\xrightarrow{\makebox[\rrrr]{$\scriptstyle{#1}$}}
\hspace{-0.5\rrrr }\hspace{-1.1 em}
\raisebox{- 0.5 ex}{$\sim$}\hspace{0.7\rrrr }
}}

\newcommand{\intoo}[1]{\:
\xymatrix@1{\ar@{^(->}[r]^{#1}&}\:}
\newcommand{\ontoo}[1]{\:
\xymatrix@1{\ar@{->>}[r]^{#1}&}\:}

\newtheorem{theorem}{Theorem}
\newtheorem{mthm}{Main Theorem}
\newtheorem{proposition}[theorem]{Proposition}
\newtheorem{lemma}[theorem]{Lemma}
\newtheorem{corollary}[theorem]{Corollary}

\newtheorem{remark}[theorem]{Remark}

\begin{document}

\title{Moduli spaces of rank 2 instanton sheaves on the projective space}

\author{Marcos Jardim}
\address{IMECC - UNICAMP \\
Departamento de Matem\'atica \\ Rua S\'ergio Buarque de Holanda, 651 \\
13083-970 Campinas-SP, Brazil}
\email{jardim@ime.unicamp.br}

\author{Mario Maican}
\address{Institute of Mathematics of the Romanian Academy, Calea Grivitei 21, Bucharest 010702, Romania}
\email{maican@imar.ro}

\author{Alexander S. Tikhomirov}
\address{Department of Mathematics\\
National Research University
Higher School of Economics\\
6 Usacheva Street\\ 
119048 Moscow, Russia}
\email{astikhomirov@mail.ru}

\begin{abstract}
We study the irreducible components of the moduli space of instanton sheaves on $\p3$, that is rank 2 torsion free sheaves $E$ with $c_1(E)=c_3(E)=0$ satisfying $h^1(E(-2))=h^2(E(-2))=0$. In particular, we classify all instanton sheaves with $c_2(E)\le4$, describing all the irreducible components of their moduli space. A key ingredient for our argument is the study of the moduli space $\calt(d)$ of stable sheaves on $\p3$ with Hilbert polynomial $P(t)=dt$, which contains, as an open subset, the moduli space of rank 0 instanton sheaves of multiplicity $d$; we describe all the irreducible components of $\calt(d)$ for $d\le4$. 
\end{abstract}

\maketitle


\section{Introduction}

Instanton bundles on $\C\p3$ were introduced by Atiyah, Drinfeld, Hitchin and Manin in the late 1970's as the holomorphic counterparts, via twistor theory, to anti-self-dual connections with finite energy (instantons) on the 4-dimensional round sphere $S^4$. To be more precise, an \emph{instanton bundle of charge $n$} is a $\mu$-stable rank 2 bundle $E$ on $\p3$ with $c_1(E)=0$ and $c_2(E)=n$ satisfying the cohomological condition $h^1(E(-2))=0$; equivalently, an instanton bundle of charge $n$ is a locally free sheaf which arises as cohomology of a linear monad of the form
\begin{equation} \label{instanton monad}
0 \to n\cdot\op3(-1) \longrightarrow (2+2n)\cdot\op3 \longrightarrow n\cdot\op3(1) \to 0 .
\end{equation}

The moduli space $\cali(n)$ of such objects has been thoroughly studied in the past 35 years by various authors and it is now known to be an irreducible \cite{T1,T2}, nonsingular \cite{JV} affine \cite{CO} variety of dimension $8n-3$. 

The closure of $\cali(n)$ within the moduli space $\calm(n)$ of semistable rank 2 sheaves with Chern classes $c_1=0$, $c_2=n$ and $c_3=0$ contains non locally free sheaves which also arise as cohomology of monads of the form (\ref{instanton monad}). Such \emph{instanton sheaves} can alternatively be defined as rank 2 torsion free sheaves satisfying the cohomological conditions
$$ h^0(E(-1)) = h^1(E(-2)) = h^2(E(-2)) = h^3(E(-3)) = 0. $$
We prove that such sheaves are always stable, see Theorem \ref{instanton stability} below, so they admit a moduli space $\call(n)$ regarded as an open subset of $\calm(n)$ which, of course, contains $\cali(n)$. 

The spaces $\call(1)$ and $\call(2)$ were essentially known to be irreducible, see details in the first few paragraphs of Section \ref{L(n)} below. However, $\call(3)$ was observed to have at least two irreducible components \cite[Remark 8.6]{JMT1}, while several new components of $\call(n)$ were constructed in \cite{JMT2}. 

The main goal of this paper is to characterize the irreducible components of $\call(3)$ and $\call(4)$. We prove:

\medskip

\begin{mthm} \label{mthm1} ~~
\begin{itemize}
\item[(i)] $\call(3)$ is a connected quasi-projective variety consisting of exactly two irreducible components each of dimension 21;
\item[(ii)] $\call(4)$ is a connected quasi-projective variety consisting of exactly four irreducible components, three of dimension 29 and one of dimension 32.
\end{itemize}
\end{mthm}

For every instanton sheaf $E$, the quotient $E^{\vee\vee}/E$ is a semistable sheaf with Hilbert polynomial $d\cdot(t+2)$ (see Section \ref{SIS} below), therefore an essential ingredient for the proof of Main Theorem \ref{mthm1} is the study of the moduli space $\calt(d)$ of semistable sheaves on $\p3$ with Hilbert polynomial $P(t)=d\cdot t$. Since these spaces are also interesting in their own right, we prove:

\begin{mthm} \label{mthm2} ~~
\begin{itemize}
\item[(i)] $\calt(1)$ is an irreducible projective variety of dimension 5;
\item[(ii)] $\calt(2)$ is a connected projective variety consisting of exactly two irreducible components of dimension 8;
\item[(iii)] $\calt(3)$ is a connected projective variety consisting of exactly four irreducible components, two of dimension 12 and two of dimension 13.
\item[(iv)] $\calt(4)$ is a connected projective variety consisting of exactly eight irreducible components, four of dimension 16, two of dimension 17, one of dimension 18 and one of dimension 20.
\end{itemize} \end{mthm}

We also give a precise description of a generic point in each of the irreducible components mentioned in the statement of the theorem, see Section \ref{1d chi=0}.


\bigskip

{\bf Acknowledgements.}
MJ is partially supported by the CNPq grant number 303332/2014-0, and the FAPESP grants number 2014/14743-8 and 2016/03759-6; this work was completed during a visit to the University of Edinburgh, and he is grateful for its hospitality. MJ also thanks Daniele Faenzi and Simone Marchesi for their help in the proof of Theorem \ref{instanton stability} below.
AST was supported by a subsidy to the HSE from the Government of the Russian Federation for the implementation of the Global Competitiveness Program. AST also acknowledges the support from the Max Planck Institute for Mathematics in Bonn, where this work was finished during the winter of 2017.


\section{Stability of instanton sheaves} \label{SIS}

Recall from \cite{J-i} that a torsion free sheaf $E$ on $\p3$ is called an \emph{instanton sheaf} if $c_1(E)=0$ and the following cohomological conditions hold
$$ h^0(E(-1))=h^1(E(-2))=h^{2}(E(-2))=h^3(E(-3))=0. $$
The integer $n:=-\chi(E(-1))$ is called the charge of $E$; it is easy to check that $n=h^1(E(-1))=c_2(E)$, and that $c_3(E)=0$. The trivial sheaf $r\cdot\op3$ of rank $r$ is considered as an instanton sheaf of charge zero. In this paper, we will only be interested in rank 2 instanton sheaves.

Recall that the singular locus ${\rm Sing}(G)$ of a coherent sheaf $G$ on a nonsingular projective variety $X$ is given by
$$  {\rm Sing}(G) := \{ x\in X ~|~ G_x ~~\text{is not free over}~~ \mathcal{O}_{X,x} \} , $$
where $G_x$ denotes the stalk of $G$ at a point $x$ and $\mathcal{O}_{X,x}$ is its local ring. The following result, proved in \cite[Main Theorem]{JG}, provides a key piece of information regarding the singular loci of rank 2 instanton sheaves.

\begin{theorem}\label{jg-thm}
If $E$ is a non locally free instanton sheaf of rank $2$ on $\p3$, then
\begin{itemize}
\item[(i)] its singular locus has pure dimension $1$;
\item[(ii)] $E^{\vee\vee}$ is a (possibly trivial) locally free instanton sheaf.
\end{itemize}
\end{theorem}

\begin{remark} \label{remark 4} \rm
In fact, the quotient sheaf $Q_E:=E^{\vee\vee}/E$ is a \emph{rank 0 instanton sheaf}, in the sense of \cite[Section 6.1]{hauzer}; see also \cite[Section 3.2]{JG}. More precisely, a \emph{rank $0$ instanton sheaf} is a coherent sheaf $Q$ on $\p3$ such that $h^0(Q(-2))=h^1(Q(-2))=0$; the integer $d:=h^0(Q(-1))$ is called the \emph{degree} of $Q$.

The Hilbert polynomial of a rank 2 instanton sheaf $E$ (in fact, of any coherent sheaf on $\p3$ of rank 2 with $c_1=0$, $c_2=n$ and $c_3=0$) is given by
\begin{equation}\label{pe(t)}
P_E(t) = \frac{1}{3}(t+3)\cdot(t+2)\cdot(t+1) - n\cdot(t+2)  = 2\cdot\chi(\op3(t)) - n\cdot(t+2) .
\end{equation}
Let $n':=c_2(E^{\vee\vee})\ge 0$; it follows from the standard sequence 
\begin{equation} \label{std dual sqc}
0 \to E \to E^{\vee\vee} \to Q_E \to 0
\end{equation}
that
$$ P_{Q_E}(t)=d\cdot(t+2) ~~ {\rm where} ~~ d:=n-n' ~. $$
Note that the $d=n-n'$ is precisely the multiplicity of $Q_E$ as a rank 0 instanton sheaf. 
\end{remark}

Rank 0 instanton sheaves can be characterized in the following way.

\begin{proposition} \label{generic_case}
Every rank 0 instanton sheaf $Q$ admits a resolution of the form
\begin{equation}
\label{generic_resolution}
0 \lra d\cdot\op3(-1) \lra 2d\cdot\op3 \lra d\cdot\op3(1) \lra Q \lra 0.
\end{equation}
\end{proposition}

\begin{proof}
Consider the Beilinson spectral sequence from \cite[Section 6]{choi_chung_maican}, applied to the sheaf $Q':=Q(-2)$.
We have $\H^0(Q') = 0$, hence also $\H^0(Q' \tensor \Omega^1_\p3(1)) = 0$ and $\H^0(Q'(-1)) = 0$.
We adopt the notations of \cite[Section 6]{choi_chung_maican}.
Since $\ker(\f_5)/\Image(\f_4) = 0$, we deduce that $\H^0(Q' \tensor \Omega^2_\p3(2)) = 0$.
Thus, the bottom row of the $E^1$-term of the spectral sequence vanishes.
Since $\f_7$ is an isomorphism, we deduce that $\f_1$ is injective.
Since $\f_8$ is injective, we deduce that $\ker(\f_2) = \Image(\f_1)$.
The top row of the $E^1$-term of the spectral sequence yields the resolution
$$ 0 \lra \H^1(Q'(-1)) \tensor \op3(-3) \overset{\f_1}{\lra}
\H^1(Q' \tensor \Omega^2_\p3(2)) \tensor \op3(-2) \overset{\f_2}{\lra} $$
$$ \overset{\f_2}{\lra} \H^1(Q' \tensor \Omega^1_\p3(1)) \tensor \op3(-1) \lra Q' \lra 0. $$
We have
\[
\chi(Q' \tensor \Omega^1_\p3(1)) = -d, \qquad \chi(Q' \tensor \Omega^2_\p3(2)) = -2d, \qquad \chi(Q'(-1)) = -d,
\]
hence
\[
\h^1(Q' \tensor \Omega^1_\p3(1)) = d, \qquad \h^1(Q' \tensor \Omega^2_\p3(2)) = 2d, \qquad \h^1(Q'(-1)) = d.
\]
The above exact sequence yields (\ref{generic_resolution}).
\end{proof}

Let now examine the stability properties of instanton sheaves.

\begin{theorem}\label{instanton stability}
Every nontrivial rank 2 instanton sheaf $E$ is stable. In addition, a nontrivial instanton sheaf $E$ is $\mu$-stable if and only if its double dual $E^{\vee\vee}$ is nontrivial.
\end{theorem}

\begin{proof}
Since rank 2 instanton sheaves have no global sections \cite[Prop. 11]{J-i}, every nontrivial locally free rank 2 instanton sheaf is $\mu$-stable; therefore, if $E^{\vee\vee}$ is nontrivial, then $E$ is also $\mu$-stable. Conversely, if $E$ is $\mu$-stable, then so is $E^{\vee\vee}$, hence it must be nontrivial. 

Therefore, in order to prove the first claim of the Theorem, it is enough to consider \emph{quasi-trivial instanton sheaves}, i.e. rank 2 instanton sheaves $E$ with $E^{\vee\vee}\simeq 2\cdot\op3$; note that the multiplicity of $Q_E$ is exactly $n=c_2(E)$.

Since $E$ has no global sections, it can only be destabilized by the ideal sheaf $\IC$ of a subscheme $C\subset\p3$. Moreover, we can assume that the quotient sheaf $E/\IC$ is torsion free, thus it is also the ideal sheaf $\ID$ of another subscheme $D\subset\p3$. We obtain two exact sequences
$$ \xymatrix{ 
         & 0   \ar[d]     & & & \\
         & \IC \ar[d]     & & & \\
0 \ar[r] & E \ar[d]\ar[r] & 2\cdot\op3 \ar[r] & Q_E \ar[r] & 0 \\
         & \ID \ar[d]     & & & \\
         & 0              & & &
}$$

Taking the double dual of the top vertical morphisms we obtain, using also the Snake Lemma, the following commutative diagram
\begin{equation}\label{big diag}
\xymatrix{ 
         & 0   \ar[d]       & 0    \ar[d]             & 0   \ar[d]       & \\
0 \ar[r] & \IC \ar[d]\ar[r] & \op3 \ar[d]\ar[r]       & \OC \ar[d]\ar[r] & 0 \\
0 \ar[r] & E \ar[d]\ar[r]   & 2\cdot\op3 \ar[d]\ar[r] & Q_E \ar[d]\ar[r] & 0 \\
0 \ar[r] & \ID \ar[d]\ar[r] & \op3 \ar[d]\ar[r]       & \OD \ar[d]\ar[r] & 0 \\
         & 0                & 0                       & 0                &
} \end{equation}
Since $h^0(Q_E(-2))=0$, then also $h^0(\OC(-2))=0$, hence $C$ must have pure dimension 1. Moreover, note also that $h^1(Q_E(-2))=0$ implies $h^1(\OD(-2))=0$. 

We show that $\dim D=0$. Indeed, assume that $D$ has dimension 1. Let $U$ be the maximal 0-dimensional subsheaf of $\OD$, and set ${\mathcal O}_{D'}:=\OD/U$; clearly, $D'$ has pure dimension 1. Next, let $D'':=D'_{\rm red}$ be the underlying reduced scheme. We end up with two exact sequences
$$ 0 \to U \to \OD \to {\mathcal O}_{D'} \to 0 ~~{\rm and}~~ $$
$$ 0 \to T \to {\mathcal O}_{D'} \to {\mathcal O}_{D''} \to 0, $$
so that the vanishing of $h^1(\OD(-2))$ forces $h^1({\mathcal O}_{D''}(-2))=0$.

Still, $D''$ may be reducible, so let $D'' := D''_1 \cup\dots\cup D''_p$ be its decomposition into irreducible components. For each index $j=1,\dots,p$ we obtain a sequence:
$$ 0 \to S_j \to {\mathcal O}_{D''} \to {\mathcal O}_{D''_j} \to 0, $$
thus also $h^1({\mathcal O}_{D''_j}(-2))=0$. Let $d_j$ and $p_j$ denote the degree and arithmetic genus of $D''_j$, respectively. It follows that
$$ 0\le h^0({\mathcal O}_{D''_j}(-2))=\chi({\mathcal O}_{D''_j}(-2))=-2d_j+1-p_j $$
hence $p_j\le-2d_j+1\le-1$, which is imposible for a reduced and irreducible curve.

Now let $\delta=h^0(\OD)$ be the length of $D$; since $\deg(C)=n$, we have
$$ P_{\IC}(t) = \chi(\op3(t)) - \chi(\OC(t)) = \chi(\op3(t)) - nt + (\delta-2n) . $$
Comparing with equation (\ref{pe(t)}), we have 
\begin{equation} \label{pe-pic}
\frac{P_E(t)}{2} - P_{\IC}(t) = \frac{n}{2}t + n - \delta ,
\end{equation}
which is positive for $n$ sufficiently large, and $E$ contains no destabilizing subsheaves.
\end{proof}

As a consequence of the proof above, we also obtain the following interesting fact.

\begin{corollary}
Every rank 2 quasi-trivial instanton on $\p3$ is an extention of the ideal of a 0-dimensional scheme $D$ by the ideal of a pure 1-dimensional scheme containing $D$.
\end{corollary}

On the other hand, it is easy to check that every rank 0 instanton sheaf is semistable.

\begin{lemma}\label{rk 0 semistable}
Every rank 0 instanton sheaf is semistable.
\end{lemma}
\begin{proof}
Let $Z$ be a rank 0 instanton sheaf, and let $T$ be a subsheaf of $Z$ with Hilbert polynomial $P_Z(t)=a\cdot t+\chi(Z)$. Since $h^0(Z(-2))=0$, then $h^0(T(-2))=0$ and $-2a+\chi(Z)=-h^1(T(-2))\leq0$. It follows that 
\[
\frac{\chi(Z)}{a} \leq 2 = \frac{\chi(Q)}{d} . \qedhere
\]
\end{proof}

Clearly, not every rank 0 instanton sheaf is stable: if $Q_1$ and $Q_2$ are rank 0 instanton sheaves, then so is any extention of $Q_1$ by $Q_2$, and this cannot possibly be stable.

Conversely, there are semistable sheaves with Hilbert polynomial $dt+2d$ which are not rank 0 instanton sheaves: just consider $Q:=\calo_\Sigma(2)$ for an elliptic curve $\Sigma\hookrightarrow\p3$, so that $h^0(Q(-2))\ne0$.


\section{Moduli space of instanton sheaves} \label{L(n)}

Let $\call(n)$ denote the open subscheme of the Maruyama moduli space $\calm(n)$ of semistable rank 2 torsion free sheaves with Chern classes $c_1=0$, $c_2=n$ and $c_3=0$ consisting of instanton sheaves of charge $n$. Let also $\cali(n)$ denote the open subscheme of $\calm(n)$ consisting of locally free instanton sheaves. Finally, let $\overline{\call(n)}$ and $\overline{\cali(n)}$ denote the closures within $\calm(n)$ of $\call(n)$ and $\cali(n)$, respectively. We also consider the set $\cali^{0}(n):=\overline{\cali(n)}\cap \call(n)$, which consists of those instanton sheaves which either are locally free, or can be deformed into locally free ones.

It was shown in \cite{T1,T2} that $\cali(n)$ is irreducible for every $n>0$; its closure $\overline{\cali(n)}$ is called the \emph{instanton component} of $\calm(n)$. However, the same is not true for $\call(n)$ as soon as $n\ge 3$. Indeed, it is well known that
$$ \overline{\cali(1)}=\call(1)=\calm(1)\simeq \p5 ,$$
see for instance \cite[Section 6]{JMT2}.

The case $n=2$ has also been understood.

\begin{proposition}\label{L(2)}
$\overline{\call(2)} = \overline{\cali(2)}$. 
\end{proposition}

In particular, $\call(2)$ possesses a single irreducible component of dimension 13.

\begin{proof}
Le Potier showed in \cite{LeP} that $\calm(2)$ has exactly 3 irreducible components; according to the description of these components provided in \cite[Section 6]{JMT2}, only the instanton component $\overline{\cali(2)}$ contains instanton sheaves. 
\end{proof}

Let us now describe the irreducible components of $\call(n)$ for $n\ge 3$ introduced in \cite[Section 3]{JMT2}.

Let $\Sigma$ be an irreducible, nonsingular, complete intersection curve in $\p3$, given as the intersection of a surface of degree $d_1$ with a surface of degree $d_2$, with $1\le d_1\le d_2$; denote by $\iota:\Sigma\hookrightarrow\p3$ the inclusion morphism. Let also $L\in\Pic^{g-1}(\Sigma)$ such that $h^0(\Sigma,L)=h^1(\Sigma,L)=0$. Given a (possibly trivial) locally free instanton sheaf $F$ of charge $c\ge0$ and an epimorphism $\varphi:E\onto(\iota_*L)(2)$, the kernel $F:=\ker\varphi$ is an instanton sheaf of charge $c+d_1d_2$. Thus we may consider the set
\begin{equation} \label{c-comp's}
\calc(d_1,d_2,c) := \left\{ [E]\in\calm(c+d_1d_2) ~|~ E^{\vee\vee}\in\cali(c) ~,~
E^{\vee\vee}/E \simeq (\iota_*L)(2) \right\}
\end{equation}
as a subvariety of $\calm(c+d_1d_2)$. The following result is proved in \cite{JMT2}, cf. Theorems 15, 17 and 23.

\begin{theorem}\label{dim C thm}
For each $c\ge0$ and $1\le d_1\le d_2$ such that $(d_1,d_2)\ne(1,1),(1,2)$, $\overline{\calc(d_1,d_2,c)}$ is an irreducible component of $\calm(c+d_1d_2)$ of dimension
\begin{equation} \label{dim C}
\dim \overline{\calc(d_1,d_2,c)} = 8c-3 + \frac{1}{2}d_1d_2(d_1+d_2+4) + h,
\end{equation}
where
$$ h = \left\{ \begin{array}{l}
2{{d_1+3}\choose{3}} - 4 , ~~{\rm if}~~ d_1=d_2 \\ ~~ \\
{{d_1+3}\choose{3}} + {{d_2+3}\choose{3}} - {{d_2-d_1+3}\choose{3}} - 2, ~~{\rm if}~~ d_1<d_2
\end{array} \right. $$
\medskip
In addition, $\overline{\calc(d_1,d_2,c)}\cap\cali^{0}(c+d_1d_2)\ne
\emptyset$.
\end{theorem}

We do not know whether the families $\calc(d_1,d_2,c)$ exhaust all components of $\call(n)$, though we prove that this holds for $n=3,4$ in Sections \ref{L(3) section} and \ref{L(4) section} below, respectively.

However, we remark that the previous result allows for a partial count of the number of components of $\call(n)$. Indeed, let $\tau(n)$ denote the number of irreducible components of the union 
$$ \overline{\cali(n)} \bigcup \left( \bigcup_{d_1d_2+c=n} \overline{\calc(d_1,d_2,c)} \right). $$
To estimate $\tau(n)$, we must count the different ways in which an integer $n\ge3$ can be written as
$n=d_1d_2+c$ with $c\ge0$, and $1\le d_1\le d_2$ excluding the pairs $(d_1,d_2)=(1,1),(1,2)$. Consider the function
$$ \delta(p) = \left\{ \begin{array}{l}
\frac{1}{2}  ( d(p) + 1 ), \mbox{ if } p \mbox{ is a perfect square} \\ ~~ \\
\frac{1}{2}  d(p), \mbox{ otherwise}
\end{array} \right. $$ 
where $d(p)$ is the \emph{divisor function}, i.e. the number of divisors of a positive integer $p$, including $p$ itself. Note that $\delta(p)$ is the number of different ways in which we can write $p$ as a product $d_1d_2$ with $1\le d_1\le d_2$. Adding the instanton component, we have that the number of irreducible components of $\call_0(n)$ is given by: 
\begin{equation} \label{ell0}
\tau(n) = 1 + \sum_{p=3}^{n} \delta(p) = 
\frac{1}{2}\left( \sum_{p=3}^{n} d(p) + \left\lfloor \sqrt{n} \right\rfloor + 1 \right),
\end{equation}
since $\left\lfloor\sqrt{n}\right\rfloor - 1$ accounts for the number of perfect squares between 3 and $n$.

\begin{lemma}
Let $l(n)$ be the number of irreducible components of the moduli space of instanton sheaves of charge $n$. Then, for $n$ sufficiently large $l(n)> \frac{1}{2} n\cdot\log(n)$. 
\end{lemma}

\begin{proof}
Determining the asymptotic behaviour of the sum of divisors function is a relevant problem in Number Theory called the \emph{Dirichlet divisor problem}; indeed, it is known that
$$  \sum_{p=1}^{n} d(p) = n\cdot\log(n) + (2\gamma - 1)n + O(n^\theta), $$
where $\gamma$ denotes the Euler--Mascheroni constant, and $1/4\le \theta\le 131/416$, cf. \cite{Huxley}. Comparing with equation (\ref{ell0}), we easily obtain the desired estimate.
\end{proof}

\bigskip

Also relevant for us is a class of instanton sheaves studied in \cite{JMT1}; more precisely, for $n>0$ and each $m=1,\dots,n$, consider the subset $\cald(m,n)$ of $\calm(n)$ consisting of the isomorphism classes $[E]$ of the sheaves $E$ obtained in this way:
$$ \cald(m,n) :=
\{ [E]\in\calm(n) ~|~ [E^{\vee\vee}]\in\cali(n-m), ~~
\Gamma=\mathrm{Supp}(E^{\vee\vee}/E)\in \calr^*_{0}(m)_{E^{\vee\vee}}~,
$$
$$
~~{\rm and}~~
E^{\vee\vee}/E\simeq\calo_\Gamma(2m-1) \},
$$
where the space $\calr^*_{0}(m)_{E^{\dual\dual}}$ is decribed as follows: first, let $\calr^*_{0}(m)$ denote the space of nonsingular rational curves $\Gamma\hookrightarrow\p3$ of degree $m$ whose normal bundle  $N_{\Gamma/\p3}$ is given by $2\cdot\calo_\Gamma((2m-1)\pt)$; then, for any instanton sheaf $F$ we set
$$ \calr^*_{0}(m)_F := \{\:\Gamma\in \calr^*_{0}(m)~|~ F|_\Gamma\simeq2\cdot\calo_\Gamma\ \}. $$
One can show that for every rank 2 instanton sheaf $F$, the space $\calr^*_{0}(m)_F$ is a nonempty open subset of $\calr^*_{0}(m)$, cf. \cite[Lemma 6.2]{JMT1}.

Let $\overline{\cald(m,n)}$ denote the closure of $\cald(m,n)$ within $\calm(n)$. Note that since $E^{\vee\vee}$ is a locally free instanton sheaf of charge $n-m$, and $\calo_\Gamma(2m-1)$ is a rank 0 instanton sheaf of degree $m$, then $E$ is an instanton sheaf of charge $n$, so that $\cald(m,n)\subset\call(n)$. In fact, it is shown in \cite[Theorem 7.8]{JMT2} that $\overline{\cald(m,n)}\subset\cali^0(n)$. In addition, we prove:

\begin{proposition} \label{disjoint rat curves}
Let $\Gamma_1,\dots,\Gamma_r$ be disjoint, smooth irreducible  rational curves in $\p3$ of degrees $m_1,\dots,m_r$, respectively; set $Q:=\bigoplus_{j=1}^{r}\calo_{\Gamma_j}(-\pt)$. If $F$ is a locally free instanton sheaf of charge $c$ such that $F|_{\Gamma_j}\simeq2\cdot\calo_{\Gamma_j}$ for each $j=1,\dots,r$, and $\varphi:F\onto Q(2)$ is an epimorphism, then $[\ker\varphi]\in\cali^0(c+m_1+\cdots+m_r)$. 
\end{proposition}

The proof of the previous proposition requires the following technical lemma, proved in \cite[Lemma 7.1]{JMT1}.

\begin{lemma}\label{F,G}
Let $C$ be a smooth irreducible curve with a marked point $0$, and set $\mathbf{B}:=C \times\p3$. Let $\mathbf{F}$ and $\mathbf{G}$ be $\mathcal{O}_{\mathbf{B}}$-sheaves, flat over $C $ and such that $\mathbf{F}$ is locally free along $\mathrm{Supp}(\mathbf{G})$. Denote
$$ \mathbf{G}_t:=\mathbf{G}|_{\{t\}\times\p3} ~~{\rm and}~~
\mathbf{F}_t=\mathbf{F}|_{\{t\}\times\p3}\ \ {\rm for}\ t\in C . $$
Assume that, for each $t\in C$,
\begin{equation}\label{vanish Hi}
H^i(\inhom(\mathbf{F}_t,\mathbf{G}_t))=0,\ \ \ i\ge1.
\end{equation}
If $s:\mathbf{F}_0\to\mathbf{G}_0$ is an epimorphism, then, after possibly shrinking $C$,
$s$~extends to an epimorhism $\mathbf{s}:\mathbf{F}\twoheadrightarrow\mathbf{G}$.
\end{lemma}

\begin{proof}[Proof of Proposition \ref{disjoint rat curves}]
We argue by induction on $r$; the case $r=1$ is just the aforementioned result, namely \cite[Theorem 7.8]{JMT2}.

Let $Q':=\bigoplus_{j=1}^{r-1}\calo_{\Gamma_j}(-\pt)$, so that $Q=Q'\oplus\calo_{\Gamma_r}(-\pt)$. Let $E:=\ker\varphi$, and let $E'$ denote the kernel of the composition $F\stackrel{\varphi}{\onto} Q(2) \onto Q'(2) $. We obtain the following exact sequence:
$$ 0 \to E \to E' \stackrel{\varphi'}{\rightarrow} \calo_{\Gamma_r}((2m_r-1)\pt) \to 0. $$
By the induction hypothesis, $[E']\in\cali^0(c+m_1+\cdots+m_{r-1})$, thus one can find an affine open subset $0\in U\subset\mathbb{A}^1$ and a coherent sheaf $\mathbf{E}$ on $\p3\times U$, flat over $U$, such that $\mathbf{E}_0=E'$ and $\mathbf{E}_t$ is a locally free instanton sheaf of charge $c+m_1+\cdots+m_{r-1}$ satisfying $\mathbf{E}_t|_{\Gamma_r}\simeq 2\cdot\calo_{\Gamma_r}$ for every $t\in U\setminus\{0\}$. Setting $\mathbf{G}=:\pi^*Q'$ where $\pi:\p3\times U\to U$ is the projection onto the first factor, note that
$$ H^i(\inhom(\mathbf{E}_t,\mathbf{G}_t)) = H^i(2\cdot\calo_{\Gamma_r}((2m_r-1)\pt)) = 0 ~~{\rm for}~~ i\ge1 ~~{\rm and}~~ t\in U. $$
This claim is clear for $t\ne0$; when $t=0$, simply observe that the sequence $0\to E'\to F \to Q'(2)\to 0$ implies that $E'_{\Gamma_r}\simeq F_{\Gamma_r}$, since the support of $Q'$ is disjoint from $\Gamma_r$.

By Lemma \ref{F,G}, there exists an epimorphism $\mathbf{s}:\mathbf{E}\twoheadrightarrow\mathbf{G}$ extending $\varphi':E'\to \calo_{\Gamma_r}((2m_r-1)\pt)$, so that
$[\ker\mathbf{s}_t]\in\cald(m_r,c+m_1+\cdots+m_r)$, by construction. It then follows that $[E]\in\overline{\cald(m_r,c+m_1+\cdots+m_r)}$, hence, by \cite[Theorem 7.8]{JMT2}, $[E]\in\cali^0(c+m_1+\cdots+m_r)$, as desired.
\end{proof}

Next, we consider the following situation: let $\Sigma$ be an irreducible, nonsingular, complete intersection curve in $\p3$, given as the intersection surfaces of degrees $d_1$ and $d_2$, with $1\le d_1\le d_2$ and $(d_1,d_2)\ne(1,1),(1,2)$, and let $\Gamma$ be a smooth irreducible rational curve in $\p3$ of degree $m$ disjoint from $\Sigma$. Set $Q:=L\oplus\calo_{\Gamma}(-\pt)$ for some $L\in\Pic^{g-1}(\Sigma)$ such that $h^0(\Sigma,L)=h^1(\Sigma,L)=0$, where $g$ is the genus of $\Sigma$.

\begin{proposition} \label{disjoint curves}
If $F$ is a locally free instanton sheaf of charge $c$ such that $F|_{\Gamma}\simeq2\cdot\calo_{\Gamma}$, and $H^1(F^\vee|_{\Gamma}\otimes L(2))=0$.
If $\varphi:F\onto Q(2)$ is an epimorphism, then $[\ker\varphi]\in\overline{\calc(d_1,d_2,c+m)}$. 
\end{proposition}

\begin{proof}
The idea is the same as in the proof of Proposition \ref{disjoint rat curves}. Let $E'$ be the kernel of the composition  $F\stackrel{\varphi}{\onto}Q(2)\onto\calo_{\Gamma}((2m-1)\pt)$, so that $E:=\ker\varphi$ and $E'$ are related via the following exact sequence:
$$ 0 \to E \to E' \stackrel{\varphi'}{\rightarrow} L(2) \to 0. $$
By \cite[Theorem 7.8]{JMT2}, one can find an affine open subset $0\in U\subset\mathbb{A}^1$ and a coherent sheaf $\mathbf{E}$ on $\p3\times U$, flat over $U$, such that $\mathbf{E}_0=E'$ and $\mathbf{E}_t$ is a locally free instanton sheaf of charge $c+m$ for every $t\in U\setminus\{0\}$. 

Setting $\mathbf{G}:=\pi^*L(2)$, we must, in order to apply Lemma \ref{F,G}, check that
$$ H^i(\inhom(E_t,G_t))=0 ~~{\rm for}~~ i\ge1 ~~{\rm and}~~ t\in U. $$
Indeed, since $\dim G_t=1$, it is enough to show that $H^1(\inhom(E_t,G_t))=0$. Note that
$$ \inhom(E_0,G_0) = \inhom(E',L(2)) \simeq \inhom(F,L(2)) \simeq F^\vee|_{\Sigma}\otimes L(2) , $$
where the middle isomorphism follows from applying the functor $\inhom(\cdot,L(2))$ to the sequence 
$$ 0\to E' \to F \to \calo_{\Gamma}((2m-1)\pt) \to0, $$
also exploring the fact that $\Sigma$ and $\Gamma$ are disjoint. It follows that $H^1(\inhom(E_0,G_0))=H^1(F^\vee|_{\Sigma}\otimes L(2))=0$ by hypothesis. By semicontinuity of $h^1(\inhom(E_t,G_t))$, we can shrink $U$ to another affine open subset $U'\subset\mathbb{A}^1$, if necessary, to guarantee that $H^1(\inhom(E_t,G_t))=0$ for every $t\in U'$.

By Lemma \ref{F,G}, there exists an epimorphism $\mathbf{s}:\mathbf{F}\twoheadrightarrow\mathbf{G}$ extending $\varphi':E'\to L(2)$, so that $[\ker\mathbf{s}_t]\in\calc(d_1,d_2,c+m)$, by construction. Since $E\simeq \ker\mathbf{s}_0$, it follows that $[E]\in\overline{\calc(d_1,d_2,c+m)}$.
\end{proof}


\section{Moduli of sheaves of dimension one and Euler characteristic zero} \label{1d chi=0}

Given two integers $d$ and $\chi$, $d\ge1$, let $\calt(d,\chi)$ 
be the moduli space of semistable coherent sheaves on $\p3$ with Hilbert polynomial $P(t)=d\cdot t+\chi$. In this section, we focus on the space  $\calt(d):=\calt(d,0)$. 

Apart from its intrinsic interest, the space $\calt(d)$ is also relevant for the study of instanton sheaves, and the description of $\calt(d)$ for $d\le4$ provided in this section will be a key ingredient for the proof of the Main Theorem \ref{mthm1}. 

In addition, let $\calz(d)$ denote the set of rank 0 instanton sheaves of degree $d$ \emph{modulo S-equivalence} (which makes sense, since, by Lemma \ref{rk 0 semistable}, every rank 0 instanton sheaf is semistable). After a twist by $\op3(-2)$, $\calz(d)$ can be regarded as an open subscheme of the moduli space $\calt(d)$ consisting of those sheaves $Q$ satisfying $h^0(Q)=0$.

The space $\calt(d)$ has several distinguished subsets, which we now describe.

First, let $\calp_d \subset \calt(d)$ be the subset of planar sheaves; it is a fiber bundle over $(\PP^3)^*$ with fiber being the moduli space of semistable coherent sheaves on $\p2$ with Hilbert polynomial $P=d\cdot t$. In view of \cite[Theorem 1.1]{lepotier}, $\calp_d$ is a projective irreducible variety of dimension $d^2+4$. In particular, $\calp_d$ is closed. 

Next, consider the subsets $\calr^o_d,~ \cale^o_d \subset \calt(d)$ of sheaves supported on smooth rational curves of degree $d$, respectively, on smooth elliptic curves of degree $d$. Let $\calr_d$ and $\cale_d$ denote their closures.

Given a partition $(d_1,\ldots,d_s)$ of $d$ such that $d_1 \ge \cdots \ge d_s$, we denote by $\calt_{d_1, \ldots, d_s} \subset \calt(d)$ the locally closed subset of points of the form 
\begin{equation}\label{polystable}
[Q_1 \oplus \cdots \oplus Q_s],
\end{equation} 
where $Q_i$ gives a stable point in $\calt(d_i)$; in particular, $\calt_d$ is the open subset of stable points in $\calt(d)$. Let $\calt^o_{d_1, \ldots, d_s} \subset \MM_{d_1, \ldots, d_s}$ be the open dense subset given by the condition that $\support(Q_i)$ be mutually disjoint. Clearly,
each irreducible component of $\calt^o_{d_1, \ldots, d_s}$ is an open dense subset of an irreducible component of $\MM_d$. Hence irreducible components of $\overline{\MM}_{d_1, \ldots, d_s}$ are also irreducible components of $\MM(d)$. On the other hand, each point of $\MM(d)$ is an $S$-equivalence class of a polystable (e. g. stable) sheaf of the form (\ref{polystable}). Hence,
the following result follows.

\begin{lemma}\label{irred comp}
(i) All irreducible components of $\MM(d)$ are exhausted by the irreducible components of the union
\begin{equation}\label{union}
\underset{(d_1,...,d_s)}{\bigcup }\overline{\MM}_{d_1, \ldots, d_s},
\end{equation} 
this union being taken over all the partitions $(d_1,...,d_s)$ of $d$.\\
(ii) For a given partition $(d_1,...,d_s)$ of $d$, each irreducible component of $\overline{\MM}_{d_1, \ldots, d_s}$ is birational to a symmetric product
\[
(\XX_1 \times \cdots \times \XX_s)/\Sigma
\]
of irreducible components $\XX_i$ of $\calt_{d_i}$, where $\Sigma$ is the subgroup of the full symmetric group $\Sigma_s$ of degree $s$ generated by the transpositions $(i, j)$ for which $d_i = d_j$ and $\XX_i = \XX_j$.
\end{lemma}
\begin{proof} We have only to prove statement (ii). Indeed, let $\Sigma' \subset \Sigma_s$ be the subgroup generated by the transpositions $(i, i+1)$ for which $d_i = d_{i+1}$. We have a bijective morphism
\[
(\calt_{d_1} \times \cdots \times \calt_{d_s})/\Sigma' \lra \calt_{d_1, \ldots, d_s}, \qquad ([Q_1], \ldots, [Q_s]) \longmapsto [Q_1 \oplus \cdots \oplus Q_s],
\]
which is an isomorphism over $\calt^o_{d_1, \ldots, d_s}$, because over this set we can construct local inverse maps. Whence, the statement (ii) follows.
\end{proof}

\begin{remark} \label{irred of Td} 
Lemma \ref{irred comp} implies that the problem of finding the irreducible components of $\calt(d)$ is reduced to the problem of finding the irreducible components of $\calt_2, \ldots, \calt_d$.
\end{remark}

\begin{remark}
It also follows from Lemma \ref{irred comp} that the number of irreducible components of $\calt(d)$ is at least as large as the number of partitions of $d$, usualy denoted $p(d)$. A well-known formula by Hardy and Ramanujan gives the following asymptotic expression
$$ p(d) \sim \frac{1}{4\sqrt{3}\cdot d}\exp\left(\pi\sqrt{\frac{2d}{3}}\right). $$
Therefore, the number of irreducible components of $\calt(d)$ grows at least exponentially on $\sqrt{d}$. However, as we shall see below in the cases $d=3$ and $d=4$, $p(d)$ is just a rough underestimate of the number of irreducible components of $\calt(d)$. 
\end{remark}

Given a coherent sheaf $Q$ on $\p3$, we define $Q^\dual := {\mathcal Ext}^c (Q, \omega_{\p3})$, where $c = \operatorname{codim} (Q)$. We will use below the following general result regarding stable sheaves in $\calt(d)$.

\begin{lemma} \label{trivial_lemma}
Assume that $\F$ gives a stable point in $\calt(d)$ and that $P \in \support(\F)$ is a closed point.
Then there are exact sequences
\begin{equation}
\label{trivial_extension}
0 \lra \E \lra \F \lra \CC_P \lra 0
\end{equation}
and
\begin{equation}
\label{trivial_sequence}
0 \lra \F \lra \G \lra \CC_P \lra 0
\end{equation}
for some sheaves $\E \in \calt(d,-1)$ and $\G \in \calt(d,+1)$.
\end{lemma}

\begin{proof}
Choose a surjective morphism $\F \to \CC_P$ and denote its kernel by $\E$.
Since $\F$ is stable, $\E$ is semi-stable, so we have sequence (\ref{trivial_extension}).
According to \cite[Theorem 13]{rendiconti}, the dual sheaf $\F^\dual$ gives a stable point in $\calt(d)$.
Thus, we have an exact sequence
\[
0 \lra \E_1 \lra \F^\dual \lra \CC_P \lra 0
\]
with $\E_1 \in \calt(d,-1)$.
According to \cite[Remark 4]{rendiconti}, $\F$ is reflexive.
According to \cite[Theorem 13]{rendiconti}, the sheaf $\G = \E_1^\dual$ gives a point in $\calt(d,1)$.
Since $\F^\dual$ is pure, we can apply \cite[Proposition 1.1.10]{huybrechts_lehn} to deduce that ${\mathcal Ext}^3(\F^\dual, \omega_{\PP^3}) = 0$.
The long exact sequence of extension sheaves associated to the above exact sequence yields (\ref{trivial_sequence}).
\end{proof}

\bigskip

The goal of this section is to describe the irreducible components of $\calt(d)$ for $d\le4$. According to \cite{dedicata}, for $\F \in \calt(d)$ we have the following cohomological conditions
\begin{align*}
& \h^0 (\F) = 0 \qquad  \text{if $d = 1$ or $2$}, \\
& \h^0 (\F) \le 1 \qquad \text{if $d = 3$ or $4$}.
\end{align*}


\subsection{Moduli of sheaves of degree 1 and 2}

The case $d=1$ is easy: clearly, $\calt(1)\simeq \calr_1$, being isomorphic to the Grassmanian of lines in $\p3$.

\begin{proposition}
The moduli space $\calt(1)$ is an irreducible projective variety of dimension 4.
\end{proposition}

In addition, it is easy to see that $\calz(1)=\calt(1)$. 

\begin{proposition} \label{components_2}
The moduli space $\calt(2)$ is connected, and has two irreducible components, each of dimension 8: $\bP_2$ (which coincides with $\RR_2$) and $\overline{\MM}_{1, 1}$.
\end{proposition}

\begin{proof}
If $\F \in \MM_2$, then we have the exact sequence (\ref{trivial_sequence}) in which $\G \in \MM(2,1)$.
Thus, $\G$ is the structure sheaf of a conic curve, hence $\G$ is planar, and hence $\F$ is planar.
We conclude that $ \MM(2) = \bP_2 \cup \overline{\MM}_{1,1}$. The intersection $\bP_2 \cap \overline{\MM}_{1,1}$ consists of those points of the form $[\calo_{\ell_1}(-1)\oplus\calo_{\ell_2}(-1)]$ where $\ell_1$ and $\ell_2$ are two intersecting (and possibly coincident) lines.
\end{proof}

Note also that $\calz(2)=\calt(2)$; the fact that $\calz(2)$ consists of two irreducible components of dimension 8 should be compared with \cite[Corollary 6.12]{hauzer}, where Hauzer and Langer prove that the moduli space of \emph{framed} rank 0 instanton sheaves also consists of two irreducible components of dimension 8.


\subsection{Moduli of sheaves of degree 3}

\begin{proposition} \label{components_3}
The moduli space $\MM(3)$ has four irreducible components $\bP_3$, $\RR_3$, $\overline{\MM}_{2,1}$ and $\overline{\MM}_{1,1,1}$, of dimension 13, 13, 12, respectively, 12.
If $\F \in \MM_3$ and $\H^0(\F) \neq 0$, then $\F$ is the structure sheaf of a planar cubic curve.
\end{proposition}

\begin{proof}
By Proposition \ref{components_2} we have $\overline{\MM}_2=\bP_2$, so that in view of Lemma \ref{irred comp} we already obtain the irreducible components $\overline{\MM}_{2,1}$ and $\overline{\MM}_{1,1,1}$ of $\MM(3)$. Therefore, by Remark \ref{irred of Td}, we only have to find the irreducible components of $\MM_3$.

Thus, given $\F \in \MM_3$, take a point $P \in \support(\F)$. We then have the exact sequence (\ref{trivial_sequence}) for $\G \in \MM(3,1)$.
According to \cite[Theorem 1.1]{freiermuth_trautmann}, $\MM(3,1)$ has two irreducible components:
the subset $\bP$ of planar sheaves and the subset $\RR$ that is the closure of the set of structure sheaves of twisted cubics.
Moreover, all sheaves in $\RR \setminus \bP$ are structure sheaves of curves $R \subset \PP^3$ of degree $3$ and arithmetic genus zero.
If $\G$ is planar, then $\F$ is planar.
If $\G = \calo_R$, then $R = \support(\F)$, where the scheme-theoretic support $\support(\F)$ of the sheaf $\F$ is defined by the 0-th Fitting ideal $\F itt^0(\F)$:\ \ $\I_{R/\PP^3}=\F itt^0(\F)$.
The morphism
\[
\rho \colon \MM_3 \setminus \bP_3 \lra \RR \setminus \bP, \qquad \rho ([\F]) = [\calo_{\support(\F)}],
\]
is injective. Indeed, if $\rho ([\F_1]) = \rho ([\F_2])$, then $\support(\F_1) = \support(\F_2) = R$.
Choose a point $P \in R$.
We have exact sequences
\[
0 \lra \F_1 \lra \G_1 \lra \CC_P \lra 0, \qquad
0 \lra \F_2 \lra \G_2 \lra \CC_P \lra 0,
\]
with $\G_1, \G_2 \in\MM(3,1)$.
Clearly, $\G_1$ and $\G_2$ are both isomorphic to $\calo_R$, hence $\F_1$ and $\F_2$ are both isomorphic to the ideal sheaf $\I_{P, R}$ of $P$ in $R$.
The image of $\rho$ is a constructible set of the irreducible variety $\RR \setminus \bP$ and contains an open subset of $\RR \setminus \bP$,
namely the subset given by the condition that $R$ be irreducible.
Indeed, if $R$ is irreducible, then it is easy to check that $\I_{P, R}$ is stable; we have $\rho ([\I_{P, R}]) = [\calo_R]$.
We deduce that $\MM_3 \setminus \bP_3$ is irreducible.
It follows that $\RR_3^o$ is dense in $\MM_3 \setminus \bP_3$.
Thus, $\MM_3$ has two irreducible components, hence $\MM(4)$ has the four irreducible components announced in the proposition.

Assume now that $\H^0(\F) \neq 0$. Then $\F$ cannot be isomorphic to $\I_{P, R}$, hence $\F$ is planar. Take a non-zero morphism $\calo \to \F$. This morphism factors through an injective morphism $\calo_C \to \F$, where $C$ is a planar curve. The semi-stability of $\F$ implies that $C$ is a cubic. Comparing Hilbert polynomials, we see that $\calo_C \to \F$ is an isomorphism.
\end{proof}



\subsection{Moduli of sheaves of degree 4}

\begin{proposition} \label{components_4}
The moduli space $\MM(4)$ has eight irreducible components:
$\bP_4$, $\EE_4$, $\RR_4$, $\overline{\MM}_{2,2}$, $\overline{\MM}_{2,1,1}$, $\overline{\MM}_{1,1,1,1}$ and two irreducible components of $\MM_{3,1}$
that are birational to $\bP_3 \times \MM_1$, respectively, to $\RR_3 \times \MM_1$.
Their dimensions are, respectively, $20$, $18$, $16$, $16$, $16$, $16$, $17$, $17$.
\end{proposition}

\begin{proof} By Propositions \ref{components_2} and \ref{components_3} and Lemma \ref{irred comp} we already have 5 irreducible components of $\MM(4)$ which are 
$\overline{\MM}_{2,2}$, $\overline{\MM}_{2,1,1}$, $\overline{\MM}_{1,1,1,1}$ and two irreducible components of $\MM_{3,1}$ that are birational to $\bP_3 \times \MM_1$, respectively, to $\RR_3 \times \MM_1$.  Therefore by Remark \ref{irred of Td} we have only to find irreducible components of $\MM_4$. Thus, given $\F \in \MM_4$, take a point $P \in \support(\F)$. We then have the exact sequence (\ref{trivial_sequence}) for $\G \in \MM(4,1)$.
According to \cite[Theorem 4.12]{choi_chung_maican}, $\MM(4,1)$ has three irreducible components:
the subset $\bP$ of planar sheaves, the subset $\RR$ that is the closure of the set of structure sheaves of rational quartic curves, and the set $\EE$
that is the closure of the set of sheaves of the form $\calo_E(P')$, where $E$ is a smooth elliptic quartic curve and $P' \in E$.
If $\G \in \bP$, then $\F \in \bP_4$.
The sheaves in $\RR \setminus (\bP \cup \EE)$ are structure sheaves of quartic curves of arithmetic genus zero.
The sheaves in $\EE \setminus \bP$ are supported on quartic curves of arithmetic genus $1$.
Let $\MM_{4, \rat} \subset \MM_4$ be the subset of sheaves whose support is a quartic curve of arithmetic genus zero.
As in Proposition \ref{components_3}, we can construct an injective dominant morphism
\[
\rho \colon \MM_{4, \rat} \lra \RR \setminus (\bP \cup \EE), \qquad \rho [\F] = [\calo_{\support(\F)}].
\]
It follows that $\MM_{4, \rat}$ is irreducible, hence $\MM_{4, \rat} \subset \RR_4$.
To finish the proof of the proposition we need to show that $\MM_4 \setminus (\bP_4 \cup \MM_{4, \rat})$ is contained in $\EE_4$.

According to \cite{space_quartics}, discussion after Proposition 8, the sheaves $\G$ in $\EE \setminus \bP$ are of two kinds:
\begin{enumerate}
\item[(i)] $\calo_E(P')$ for a curve $E$ of arithmetic genus $1$ given by an ideal of the form $(q_1, q_2)$, where $q_1$, $q_2$ are quadratic forms,
and $P' \in E$. Notice that $\Ext^1_{\calo_{\PP^3}}(\CC_{P'}, \calo_E) \isom{}\CC$, so the notation $\calo_E(P')$ is justified.
Also note that $\calo_E$ is stable;
\item[(ii)] non-planar extensions of the form
\[
0 \lra \calo_L(-1) \lra \G \lra \calc \lra 0,
\]
where $L$ is a line and $\calc$ gives a point in $\MM_H(3,1)$ for a plane $H$ possibly containing $L$. (Here and below we use the notation $\MM_S(d,\chi)$ for the moduli space of 1-dimensional sheaves on a given surface $S$ in $\PP^3$ with Hilbert polynomial $P(t)=dt+\chi$. We also set $\MM_S(d):=\MM_S(d,0)$.) 
\end{enumerate}

\noindent \\
\emph{Claim 1}. Case (ii) is unfeasible.

\noindent \\
Assume, firstly, that $P \in H$.
Tensoring (\ref{trivial_sequence}) with $\calo_H$, we get the exact sequence
\[
\F_{| H} \lra \G_{| H} \overset{\alpha}{\lra} \calc_P \lra 0.
\]
Thus, $\Ker(\alpha)$ is a quotient sheaf of $\F$ of slope zero. This contradicts the stability of $\F$.
Assume, secondly, that $P \notin H$.
According to \cite[Sequence (10)]{space_quartics}, we have an exact sequence
\[
0 \lra \E \lra \G \lra \calo_L \lra 0
\]
for some sheaf $\E \in \MM_H (3)$. The composite map $\E \to \G \to \calc_P$ is zero, hence $\E$ is a subsheaf of $\F$.
This contradicts the stability of $\F$ and proves Claim 1.

\noindent \\
It remains to deal with the sheaves from (i). We have one of the following possibilities:
\begin{enumerate}
\item[(a)] $E$ is contained in a smooth quadric surface $S$;
\item[(b)] $E$ is contained in an irreducible cone $\Sigma$ but not in a smooth quadric surface;
\item[(c)] $\operatorname{span}\{ q_1, q_2 \}$ contains only reducible quadratic forms and $q_1$, $q_2$ have no common factor.
\end{enumerate}

\noindent \\
\emph{Claim 2}. In case (a), $\F$ belongs to $\EE_4$.

\noindent \\
Notice that $\F \in \MM_S(4)$.
According to \cite[Proposition 7]{ballico_huh}, $\MM_S(4)$ has five disjoint irreducible components $\MM_S (p, q, 4)$, where $(p, q)$ is the type of the support of the 1-dimensional sheaf w.r.t. $\mathrm{Pic}S$. Clearly, $\F \in \MM_S (2, 2, 4)$.
Thus, $\F$ is a limit of sheaves in $\MM_S (2, 2, 4)$ supported on smooth curves of type $(2, 2)$, hence $\F \in \EE_4$.

\noindent \\
It remains to deal with cases (b) and (c). Next we reduce further to the case when $P = P'$.
Notice that, if $P \neq P'$, then $\F \simeq \calo_E(P') \tensor (\calo_E(P))^\dual$, hence the notation $\F = \calo_E(P' - P)$ is justified.

\noindent \\
\emph{Claim 3}. Assume that $\F = \calo_E(P' - P)$ for an elliptic quartic curve $E$ and distinct closed points $P', P \in E$.
Then $\F$ belongs to $\EE_4$.

\noindent \\
Let $Z_1, \ldots, Z_m$, denote the irreducible components of $E$. Fix $i, j \in \{ 1, \ldots, m \}$.
Consider the locally closed subset $\XX \subset \EE \times \EE$ of pairs $([\calo_{E'}(P_1)], [\calo_{E'}(P_2)])$, where $E'$ is a quartic curve
of arithmetic genus $1$ whose ideal is generated by two quadratic polynomials, and $P_1$, $P_2$ are distinct points on $E'$
such that $P_1 \notin \cup_{k \neq i} Z_k$, $P_2 \notin \cup_{k \neq j} Z_k$.
Consider the morphisms
\[
\xi \colon \XX \lra \MM(4), \qquad ([\calo_{E'}(P_1)], [\calo_{E'}(P_2)]) \longmapsto [\calo_{E'}(P_1 - P_2)],
\]
\[
\sigma \colon \XX \lra \Hilb_{\PP^3}(4t), \qquad ([\calo_{E'}(P_1)], [\calo_{E'}(P_2)]) \longmapsto E',
\]
where $\Hilb_{\PP^3}(4t)$ is the Hilbert scheme of subschemes of $\PP^3$ with Hilbert polynomial $P(t)=4t$.
According to \cite[Examples 2.8 and 4.8]{chen_nollet}, $\Hilb_{\PP^3}(4t)$ consists of two irreducible components, denoted $\HH_1$ and $\HH_2$.
The generic member of $\HH_1$ is a smooth elliptic quartic curve.
The generic member of $\HH_2$ is the disjoint union of a planar quartic curve and two isolated points.
Note that $\HH_2$ lies in the closed subset $\{ E' \mid \ \h^0(\calo_{E'}) \ge 3 \}$.
Since $E$ lies in the complement of this subset, we deduce that $E \in \HH_1$.
It follows that there exists an irreducible quasi-projective curve $\Gamma \subset \Hilb_{\PP^3}(4t)$ containing $E$,
such that $\Gamma \setminus \{ E \}$ consists of smooth elliptic quartic curves (see the proof of \cite[Proposition 12]{space_quartics}).
The fibers of the map $\sigma^{-1} (\Gamma) \to \Gamma$ are irreducible of dimension $2$.
By \cite[Theorem 8, p. 77]{shafarevich}, we deduce that $\sigma^{-1}(\Gamma)$ is irreducible.
Thus, $\xi (\sigma^{-1}(\Gamma))$ is irreducible. This set contains $[\calo_{E}(P'-P)]$
for $P' \in Z_i \setminus \cup_{k \neq i} Z_k$, $P \in Z_j \setminus \cup_{k \neq j} Z_k$.
The generic member of $\xi (\sigma^{-1}(\Gamma))$ is a sheaf supported on a smooth elliptic quartic curve.
We conclude that $[\calo_E(P' - P)] \in \EE_4$. Since $i$ and $j$ are arbitrary, the result is true for all $P'$, $P$ closed points on $E$.

\noindent \\
\emph{Claim 4}. In case (c), $E$ is a quadruple line supported on a line $L$. More precisely, there are three distinct planes
$H$, $H'$, $H''$ containing $L$, such that $E = (H \cup H') \cap (2H'')$.

\noindent \\
The claim will follow if we can show that there are linearly independent forms $u, v \in V^*$ such that $q_1, q_2 \in \CC[u, v]$.
Indeed, in this case $(q_1, q_2)$ has the normal form $(uv, (u+v)^2)$.
We argue by contradiction.
Assume that $q_1 = XY$ and $q_2 = Zl$. 
Consider first the case when $l = aX + bY + cZ$.
We will find $\lambda \in \CC$ such that $f = XY + \lambda Z l$ is irreducible,
which is equivalent to saying that
\[
\frac{\partial f}{\partial X} = Y + a\lambda Z, \qquad \frac{\partial f}{\partial Y} = X + b\lambda Z, \qquad \frac{\partial f}{\partial Z} = \lambda (aX + bY + 2cZ)
\]
have no common zero, or, equivalently,
\[
\left|
\ba{ccc}
0 & 1 & a\lambda \\
1 & 0 & b\lambda \\
a\lambda & b\lambda & 2c\lambda
\ea
\right| \neq 0.
\]
We have reduced to the inequality $2 a b \lambda^2 - 2c\lambda \neq 0$. If $c \neq 0$ we can find a solution.
If $c = 0$, then $a b \neq 0$, otherwise $q_1$ and $q_2$ would have a common factor, and we can choose any $\lambda \in \CC^*$.
Assume now that $l = aX + bY + cZ + dW$ with $d \neq 0$. Note that $f = XY + \lambda Z l$ is irreducible if its image in
\[
\CC[X, Y, Z, W]/\langle (c-1)Z + dW \rangle \simeq \CC[X, Y, Z]
\]
is irreducible. The above isomorphism sends $f$ to $XY + \lambda Z (aX + bY + Z)$ which brings us to the case examined above.

\noindent \\
\emph{Claim 5}. In case (c), $\F$ belongs to $\EE_4$.

\noindent \\
We have ${\calo_E}_{|H} \simeq \calo_C$, ${\calo_E}_{|H'} \simeq \calo_{C'}$ for conic curves $C$ and $C'$ supported on $L$.
The kernel of the map $\calo_E \to \calo_C$ has Hilbert polynomial $2t-1$ and is stable, because $\calo_E$ is stable,
hence it is isomorphic to $\calo_{C'}(-1)$.
We have a commutative diagram
\[
\xymatrix
{
0 \ar[r] & \calo_E \ar[r] \ar[d] & \calo_E(P') \ar[r] \ar[d] & \CC_{P'} \ar[r] \ar@{=}[d] & 0 \\
0 \ar[r] & \calo_C \ar[r] & \calo_E(P')_{|H} \ar[r] & \CC_{P'} \ar[r] & 0
}
\]
in which the second row is obtained by restricting the first row to $H$.
Applying the snake lemma, we obtain the first row of the following exact commutative diagram:
\[
\xymatrix
{
0 \ar[r] & \calo_{C'}(-1) \ar[r] & \calo_E(P') \ar[r] \ar[d] & \calo_E(P')_{|H} \ar[r] \ar[d]^-{\alpha} & 0 \\
& & \CC_P \ar@{=}[r] & \CC_P
}
\]
Applying the snake lemma to this diagram, we get the exact sequence
\[
0 \lra \calo_{C'}(-1) \lra \F \lra \Ker(\alpha) \lra 0.
\]
Note that $\Ker(\alpha)$ has Hilbert polynomial $2t+1$ and is semi-stable, being a quotient of the stable sheaf $\F$.
It follows that $\Ker(\alpha) \simeq \calo_C$.
Thus, $\F$ gives a point in the set $\PP \big( \Ext^1(\calo_C, \calo_{C'}(-1)) \big)^\st$ of stable non-split extensions of $\calo_C$ by $\calo_{C'}(-1)$.

Consider the family of planes $H''_t$, $t \in \PP^1 \setminus \{ 0, \infty \}$, containing $L$ and different from $H$ and $H'$.
Denote $E_t = (H \cup H') \cap (2 H''_t)$. We have a two-dimensional family of semi-stable sheaves
\[
\{ \calo_{E_t}(P' - P'') \mid t \in \PP^1 \setminus \{ 0, \infty \}, \ P'' \in L \setminus \{ P' \} \} \subset \PP \big( \Ext^1(\calo_C, \calo_{C'}(-1))\big).
\]
This family is dense in the right-hand-side because $\Ext^1_{\calo_{\PP^3}} (\calo_C, \calo_{C'}(-1)) \simeq \CC^3$.
To prove this we use the standard exact sequence obtained from Thomas' spectral sequence
\begin{multline*}
0 \to \Ext^1_{\calo_{H'}} ({\calo_C}_{| H'}, \calo_{C'}(-1)) \to \Ext^1_{\calo_{\PP^3}} (\calo_C, \calo_{C'}(-1)) \to \\ \Hom (\Tor_1^{\calo_{\PP^3}}(\calo_C, \calo_{H'}), \calo_{C'}(-1))
\to \Ext^2_{\calo_{H'}} ({\calo_C}_{| H'}, \calo_{C'}(-1)),
\end{multline*}
see also \cite[Lemma 4.2]{choi_chung_maican}. Note that ${\calo_C}_{| H'} \simeq \calo_L$. Using Serre duality we obtain the isomorphisms
\begin{align*}
\Ext^2_{\calo_{H'}} (\calo_L, \calo_{C'}(-1)) & \simeq \Hom_{\calo_{H'}} (\calo_{C'}(-1), \calo_L(-3))^* = 0, \\
\Ext^1_{\calo_{H'}} (\calo_L, \calo_{C'}(-1)) & \simeq \Ext^1_{\calo_{H'}} (\calo_{C'}(-1), \calo_L(-3))^* \simeq \CC^2.
\end{align*}
The last isomorphism follows from the long exact sequence of extension sheaves
\begin{multline*}
0 = \Hom (\calo_{H'}(-1), \calo_L(-3)) \lra \Hom (\calo_{H'}(-3), \calo_L(-3)) \simeq \H^0 (\calo_L) \simeq \CC \\
\lra \Ext^1_{\calo_{H'}} (\calo_{C'}(-1), \calo_L(-3)) \lra \\
\Ext^1_{\calo_{H'}} (\calo_{H'}(-1), \calo_L(-3)) \simeq \H^1(\calo_L(-2)) \simeq \CC \lra \Ext^1_{\calo_{H'}} (\calo_{H'}(-3), \calo_L(-3)) = 0
\end{multline*}
derived from the short exact sequence
\[
0 \lra \calo_{H'}(-3) \lra \calo_{H'}(-1) \lra \calo_{C'}(-1) \lra 0.
\]
Choose linear forms $u$ and $u'$ defining $H$ and $H'$. Restricting the standard resolution
\[
0 \lra \calo(-3) \xrightarrow{\tiny \left[ \!\! \ba{l} -u \\ (u')^2 \ea \!\!\! \right]} \calo(-2) \oplus \calo(-1)
\xrightarrow{\tiny \left[ \!\! \ba{cc} (u')^2 & \!\! u \ea \!\! \right]} \calo \lra \calo_C \lra 0
\]
to $H'$, we see that $\Tor_1^{\calo_{\PP^3}}(\calo_C, \calo_{H'})$ is isomorphic to the cohomology of the complex
\[
\calo_{H'}(-3) \xrightarrow{\tiny \left[ \!\! \ba{l} - u_{| H'} \\ \phantom{-} 0 \ea \!\! \right]} \calo_{H'}(-2) \oplus \calo_{H'}(-1)
\xrightarrow{\tiny \left[ \!\! \ba{cc} 0 & \!\! u_{| H'} \ea \!\! \right]} \calo_{H'}
\]
that is, to $\calo_L(-2)$. Using the fact that $\calo_{C'}(-1)$ and $\calo_L(-2)$ are reflexive, we have the isomorphisms
\[
\Hom (\calo_L(-2), \calo_{C'}(-1)) \simeq \Hom (\calo_{C'}(-1)^\dual, \calo_L(-2)^\dual) \simeq \Hom (\calo_{C'}, \calo_L) \simeq \CC.
\]
The above discussion shows that $[\F]$ is a limit of points in $\M_{\PP^3}(4m)$ of the form $[\calo_{E_t}(P' - P'')]$, with $P' \neq P''$.
Claim 5 now follows from Claim 3.

\noindent \\
It remains to consider sheaves $\F$ given by sequence (\ref{trivial_sequence}) in which $\G = \calo_E(P)$ and $E$ is as at (b).
We reduce further to the case when $E$ has no regular points.

\noindent \\
\emph{Claim 6}. Assume that $E$ has a regular point. Then $\F \simeq \calo_E$, hence $\F$ belongs to $\EE_4$.

\noindent \\
The proof of the claim is obvious because $P$ in sequence (\ref{trivial_sequence}) can be chosen arbitrarily on $E$.
We choose $P \in \operatorname{reg}(E)$. The kernel of the map $\calo_E(P) \to \CC_P$ is $\calo_E$. Note that $E$ belongs to the irreducible
component $\HH_1$ of $\Hilb_{\PP^3}(4t)$, hence it is the limit of smooth elliptic quartic curves.

\noindent \\
\emph{Claim 7}. Let $E \subset \PP^3$ be a quartic curve of arithmetic genus $1$ which is contained in an irreducible cone $\Sigma$,
but not in a smooth quadric surface. Assume that $E$ has no regular points. Then we have one of the following two possibilities:
\begin{enumerate}
\item[(b1)] $E = \Sigma \cap (H \cup H')$, where $H$, $H'$ are distinct planes each intersecting $\Sigma$ along a double line;
\item[(b2)] $E = \Sigma \cap (2H)$, where $H$ is a plane intersecting $\Sigma$ along a double line.
\end{enumerate}

\noindent \\
To fix notations assume that $\Sigma$ has vertex $O$ and base a conic curve $\Gamma$ contained in a plane $\Pi$.
Assume first that $E = \Sigma \cap \Sigma'$ for $\Sigma'$ another irreducible cone.
If $\Sigma$ and $\Sigma'$ have distinct vertices, then $E$ has regular points.
Thus, $\Sigma'$ has vertex $O$ and base an irreducible conic curve $\Gamma'$ contained in $\Pi$.
Since $E$ has no regular points, $\Gamma \cap \Gamma'$ is the union of two double points $Q_1$ and $Q_2$.
Now $E$ is the cone with vertex $O$ and base $Q_1 \cup Q_2$, so $E$ is as at (b1).

Assume next that $E = \Sigma \cap (H \cup H')$ for distinct planes $H$ and $H'$.
If $H$ or $H'$ does not contain $O$, then $E$ has regular points.
If $H$ or $H'$ is not tangent to $\Gamma$, then $E$ has regular points.
We deduce that $E$ is as in (b1).

Assume, finally, that $E = \Sigma \cap (2H)$ for a double plane $2H$.
If $O \notin H$, then it can be shown that $E$ is contained in a smooth quadric surface.
Indeed, assume that $\Sigma$ has equation $X^2 + Y^2 + Z^2 = 0$ and $H$ has equation $W=0$.
Then $E$ is contained in the smooth quadric surface with equation $X^2 + Y^2 + Z^2 + W^2 = 0$.
Thus, $O \in H$. If $\Gamma \cap H$ is the union of two distinct points, then $\Gamma \cap (2H)$ is the union of two double points
$Q_1$ and $Q_2$ and $E$ is as in (b1).
If $\Gamma \cap H$ is a double point, then $E$ is as in (b2).

\noindent \\
\emph{Claim 8}. In case (b1), $\F$ belongs to $\EE_4$.

\noindent \\
We have $\calo_{E | H} \simeq \calo_C$, $\calo_{E | H'} \simeq \calo_{C'}$ for conic curves $C$, $C'$ supported on lines $L$, respectively, $L'$.
Assume that $P \in L$ and choose a point $P' \in L$ not necessarily distinct from $P$.
Let $\F' \in \MM_4$ be given by the exact sequence
\[
0 \lra \F' \lra \calo_E(P') \lra \CC_P \lra 0.
\]
As in the first paragraph in the proof of Claim 5, we see that $\F'$ gives a point in the set $\PP\big( \Ext^1(\calo_C, \calo_{C'}(-1))\big)^\st$.
We have $\dim \Ext^1_{\calo_{\PP^3}}(\calo_C, \calo_{C'}(-1)) \le 2$. Indeed, start with the exact sequence
\begin{align*}
0 \to \Ext^1_{\calo_{H'}}(\calo_{C | H'}, \calo_{C'}(-1)) & \to \Ext^1_{\calo_{\PP^3}}(\calo_C, \calo_{C'}(-1)) \to \\ & \to \Hom(\Tor_1^{\calo_{\PP^3}}(\calo_C, \calo_{H'}), \calo_{C'}(-1)).
\end{align*}
The group on the second line vanishes because $\Tor_1^{\calo_{\PP^3}}(\calo_C, \calo_{H'})$ is supported on $O$ while $\calo_{C'}(-1)$ has no zero-dimensional torsion. It follows that
$$ \Ext^1_{\calo_{\PP^3}}(\calo_C, \calo_{C'}(-1)) \simeq \Ext^1_{\calo_{H'}}(\calo_{C | H'}, \calo_{C'}(-1)). $$
The sheaf $\calo_{C | H'}$ is the structure sheaf of a double point supported on $O$, hence we have the exact sequence
\[
\CC \simeq \Ext^1_{\calo_{H'}}(\CC_O, \calo_{C'}(-1)) \to \Ext^1_{\calo_{H'}}(\calo_{C | H'}, \calo_{C'}(-1)) \to \Ext^1_{\calo_{H'}}(\CC_O, \calo_{C'}(-1)) \simeq \CC
\]
from which we get our estimate on the dimension of the middle vector space.

The one-dimensional family $\calo_E(P' - P)$, $P' \in L \setminus \{ P \}$, is therefore dense in $\PP\big( \Ext^1(\calo_C, \calo_{C'}(-1))\big)^\st$,
hence, in view of Claim 3, $\F$ is a limit of sheaves in $\EE_4$. We conclude that $\F \in \EE_4$.

\noindent \\
\emph{Claim 9}. In case (b2), $\F$ belongs to $\EE_4$.

\noindent \\
Let $L$ be the reduced support of $\Sigma \cap H$. We have $\calo_{E | H} \simeq \calo_C$ for a conic curve supported on $L$.
Choose a point $P' \in L$ not necessarily distinct from $P$ and let $\F' \in \MM_4$ be given by the exact sequence
\[
0 \lra \F' \lra \calo_E(P') \lra \CC_P \lra 0.
\]
As in the first paragraph of the proof of Claim 5, we see that $\F'$ gives a point in the set $\PP\big( \Ext^1(\calo_C, \calo_C(-1))\big)^\st$.
We have $\dim \Ext^1_{\calo_{\PP^3}}(\calo_C, \calo_C(-1)) = 5$. This follows from the exact sequence
\begin{multline*}
0 \to \Ext^1_{\calo_H}(\calo_C, \calo_C(-1)) \to \Ext^1_{\calo_{\PP^3}}(\calo_C, \calo_C(-1)) \to \\
\Hom(\Tor_1^{\calo_{\PP^3}}(\calo_C, \calo_H), \calo_C(-1)) \to \Ext^2_{\calo_H}(\calo_C, \calo_C(-1)).
\end{multline*}

From Serre duality we get
\[
\Ext^2_{\calo_H}(\calo_C, \calo_C(-1)) \simeq \Hom_{\calo_H}(\calo_C(-1), \calo_C(-3))^* \simeq \H^0(\calo_C(-2))^* = 0.
\]
We have $\Tor_1^{\calo_{\PP^3}}(\calo_C, \calo_H) \simeq \calo_C(-1)$ hence $\Hom(\Tor_1^{\calo_{\PP^3}}(\calo_C, \calo_H), \calo_C(-1)) \simeq \CC$.
Applying the functor $\Hom(\cdot,\calo_C(-1))$ to the short exact sequence
\[
0 \lra \calo_H(-2) \lra \calo_H \lra \calo_C \lra 0
\]
we obtain the following exact sequence, 
\begin{multline*}
0 \to \Hom(\calo_H(-2), \calo_C(-1)) \simeq \H^0(\calo_C(1)) \simeq \CC^3 \\
\to \Ext^1_{\calo_H}(\calo_C, \calo_C(-1)) \to \\
\Ext^1_{\calo_H}(\calo_H, \calo_C(-1)) \simeq \H^1(\calo_C(-1)) \simeq \CC \to 0,
\end{multline*}
since $\Hom(\calo_H, \calo_C(-1)) \simeq \H^0(\calo_C(-1)) = 0$, and 
$\Ext^1_{\calo_H}(\calo_H(-2), \calo_C(-1)) \simeq \H^1(\calo_C(1)) = 0$.

Denote $Q = L \cap \Pi$. We have a three-dimensional family $\Gamma_t$ of irreducible conic curves in $\Pi$ that contain $Q$
and are tangent to $H$. Let $\Sigma_t$ be the cone with vertex $O$ and base $\Gamma_t$.
Put $E_t = \Sigma_t \cap (2H)$. The four-dimensional family $\calo_{E_t}(P' - P)$, $P' \in L \setminus \{ P \}$ is dense in
$\PP\big( \Ext^1(\calo_C, \calo_C(-1))\big)^\st$, hence, in view of Claim 3, $\F$ is the limit of sheaves in $\EE_4$.
We conclude that $\F \in \EE_4$.
\end{proof}

The proof of Main Theorem \ref{mthm2} is finally complete.



\section{Components and connectedness of $\call(3)$} \label{L(3) section}

We are now ready to prove that the moduli space of rank 2 instanton sheaves of charge 3 on $\p3$ is connected and has precisely two irreducible components. Indeed, the two components of $\overline{\call(3)}$ have already been identified above:
\begin{itemize}
\item[(I)] $\overline{\cali(3)}$, whose generic point corresponds to locally free instanton sheaves.
\item[(II)] $\overline{\calc(1,3,0)}$, whose generic point corresponds to instanton sheaves $E$ fitting into exact sequences of the form
\begin{equation}\label{C(1,3,0)}
0 \to E \to 2\cdot \op3 \to \iota_*L(2) \to 0
\end{equation}
where $\iota:\Sigma\hookrightarrow\p3$ is the inclusion of a nonsingular plane cubic $\Sigma$, and $L\in{\rm Pic}^0(\Sigma)$ is such that $h^0(\Sigma,L)=0$.
\end{itemize}
Both components have dimension 21; this is classical result for the component $\cali^{0}(3)$, while the dimension of $\calc(1,3,0)$ is given by Theorem \ref{dim C thm}. In addition, this same result also guarantees that the union $\cali^{0}(3)\cup \calc(1,3,0)$ is connected.

Therefore, our task is to prove that $\overline{\call(3)}$ has no other irreducible components, i.e. that every instanton sheaf of charge 3 can be deformed either into a locally free instanton sheaf, or into an instanton sheaf given by a sequence of the form (\ref{C(1,3,0)}). 

So let $E$ be a non locally free instanton sheaf of charge 3, and let $Q_E:=E^{\vee\vee}/E$ be the corresponding rank 0 instanton sheaf; let $d_E$ denote the degree of $Q_E$. There are three possibilities to consider: $d_E=1$, $d_E=2$ and $d_E=3$.

The first possibility is easy to deal with: if $d_E=1$, then $Q_E=\calo_\ell(1)$, where $\ell\hookrightarrow\p3$ is a line in $\p3$. It follows that $E$ fits into an exact sequence of the form
$$ 0 \to E \to F \to \calo_\ell(1) \to 0 ~, $$
where $F$ is a locally free instanton sheaf of charge 2. However, \cite[Proposition 7.2]{JMT1} ensures that $E$ can be deformed in a ('t Hooft) locally free instanton sheaf of charge 3. In other words, if $d_E=1$, then $E$ lies within $\cali^0(3)$.

Now, if $d_E=2$, then, since $Q_E$ is semistable and by Proposition \ref{components_2} above, one can find an affine open subset $0\in U\subset\bA^1$ and a coherent sheaf $\mathbf{G}$ on $\p3\times U$ such that $G_0=Q_E$ and, for $u\ne 0$, 
\begin{itemize}
\item[\bf{(i)}] either $G_u=\calo_\Gamma(3\mathrm{pt})$, where $\Gamma$ is a nonsingular conic in $\p3$; 
\item[\bf{(ii)}] or $G_u=\calo_{\ell_1}(1)\oplus\calo_{\ell_2}(1)$ where $\ell_1$ and $\ell_2$ are skew lines in $\p3$. 
\end{itemize}

Since $d_E=2$, $E^{\vee\vee}$ is a locally free instanton sheaf of charge 1 (a.k.a. a null-correlation bundle); we set $N:=E^{\vee\vee}$ denote such sheaf. Take $\mathbf{F}:=\pi^*N$, where $\pi:\p3\times U \to \p3$ is the projection onto the first factor. Let $s:N\twoheadrightarrow Q_E$ be the epimorphism given by the standard sequence (\ref{std dual sqc}). For every $u\in U$, the sheaf $\inhom(F_u,G_u)\simeq N\otimes G_u$ is supported in dimension 1, thus clearly $H^i(\inhom(F_u,G_u))=0$ for $i=2,3$. For $u\ne 0$ we can, after possibly shrinking $U$, assume that either $N|_\Gamma\simeq 2\cdot\calo_\Gamma$ or $N|_{\ell_1}\simeq 2\cdot\calo_{\ell_1}$ and $N|_{\ell_2}\simeq 2\cdot\calo_{\ell_2}$; in both situations, it is easy to check that $H^1(\inhom(F_u,G_u))=0$. Finally, for $u=0$, we twist the resolution of $Q_E$
$$ 0 \to 2\cdot \op3(-1) \stackrel{\alpha}{\longrightarrow} 4\cdot \op3 \stackrel{\beta}{\longrightarrow} 2\cdot \op3(1) \to Q_E \to 0 $$
by $N$ and check that $H^1(\inhom(F_0,G_0))=H^1(N\otimes Q_E)\simeq H^2(N\otimes B)=0$, where $B:=\im\beta$.

Therefore, it follows from Lemma \ref{F,G} that there exists an epimorphism $\mathbf{s}:\mathbf{F}\twoheadrightarrow\mathbf{G}$ extending $s:N\twoheadrightarrow Q_E$. Let $\mathbf{E}:=\ker\mathbf{s}$; clearly, $E_0:=\mathbf{E}|_{\{0\}\times\p3}=E$. For $u\ne 0$, $E_u$ fits into the exact sequence
$$ 0 \to E_u \to N \to G_u \to 0. $$

In the case {\bf (i)} described above, $G_u$ lies within $\cald(2,3)$ for $u\ne0$, hence $E=E_0$ lies within $\overline{\cald(2,3)}$, which is contained in $\overline{\cali(3)}$ by \cite[Theorem 7.8]{JMT1}. In other words, $E_0$ can be deformed into a locally free instanton sheaf of charge 3, thus it lies within $\cali^0(3)$.

In the case {\bf (ii)}, Propostion \ref{disjoint rat curves} also implies that $[E_0]\in\cali^0(3)$.

An argument similar to the one used in the proof of \cite[Proposition 7.2]{JMT1} works to show that $E$ can be deformed into a locally free ('t Hooft) instanton sheaf.

Summing up, we conclude that if $d_E=2$, then $E$ lies within $\cali^0(3)$.

Finally, consider $d_E=3$, so that $E^{\vee\vee}=2\cdot\op3$. Since $Q_E$ is semistable, it follows from Proposition \ref{components_3} that one can find an affine open subset $0\in U\subset\bA^1$ and a coherent sheaf $\mathbf{G}$ on $\p3\times U$ such that $G_0=Q_E$ and, for $u\ne 0$, 
\begin{itemize}
\item[{\bf (i)}] either $G_u=\calo_\Delta(5\mathrm{pt})$, where $\Delta$ is a nonsingular twisted cubic in $\p3$;
\item[{\bf (ii)}] or $G_u=\calo_\Gamma(3\mathrm{pt})\oplus
\calo_{\ell}(1)$, where $\Gamma$ is a nonsingular conic and $\ell$ is a line disjoint from $\ell$;
\item[{\bf (iii)}] or $G_u=\calo_{\ell_1}(1)\oplus\calo_{\ell_2}(1)\oplus\calo_{\ell_3}(1)$ where $\ell_j$ are 3 skew lines in $\p3$;
\item[{\bf (iv)}] or $G_u=L(2)$, where $L\in\Pic^0(\Sigma)$, for some nonsingular plane cubic $\Sigma$ in $\p3$. 
\end{itemize}

Now set $\mathbf{F}:2\cdot\pi^*\op3$. Note that $H^i(\inhom(F_u,G_u))=H^i(2\cdot G_u)$, and this vanishes for $i=1,2,3$ in all of the four cases outlined above for $u\ne0$; for $u=0$, $H^i(G_0)=H^i(Q_E)$ and this vanishes by dimension of $Q_E$ when $i=2,3$, and by the vanishing of $h^1(Q_E(-2))$ when $i=1$.

We complete the argument as before; again, it follows from Lemma \ref{F,G} that there exists an epimorphism $\mathbf{s}:\mathbf{F}\twoheadrightarrow\mathbf{G}$ extending the epimorphism $s:2\cdot\op3\twoheadrightarrow Q_E$ obtained from the standard sequence (\ref{std dual sqc}) for $E$. Let $\mathbf{E}:=\ker\mathbf{s}$; clearly, $E_0:=\mathbf{E}|_{\{0\}\times\p3}=E$. For $u\ne 0$, $E_u$ fits into the exact sequence
$$ 0 \to E_u \to 2\cdot\op3 \to G_u \to 0. $$

In the cases {\bf (i)} through {\bf (iii)}, we know from \cite[Theorem 7.8]{JMT1} and Proposition \ref{disjoint rat curves} above that $[E_0]\in\overline{\cald(3,3)}$, thus also $[E_0]\in\cali^0(3)$.

In the case {\bf (iv)}, $E_u$ lies within $\calc(1,3,0)$, by definition. This completes the proof of the first part of Main Theorem \ref{mthm1}.


\section{Components and connectedness of $\call(4)$} \label{L(4) section}

In this section we prove the second part of Main Theorem \ref{mthm1}, i.e. we enumerate the irreducible components of $\call(4)$ and show that $\call(4)$ is connected. Note that we already know from Theorem \ref{dim C thm} four irreducible components of $\overline{\call(4)}$:

\begin{itemize}
\item[(I)] $\overline{\cali(4)}$, whose generic point corresponds to locally free instanton sheaves;
\item[(II)] $\overline{\calc(1,3,1)}$, whose generic point corresponds to instanton sheaves $E$ fitting into exact sequences of the form
\begin{equation}\label{C(1,3,1)}
0 \to E \to N\to \iota_*L(2) \to 0,
\end{equation}
where $\iota:\Sigma\hookrightarrow\p3$ is the inclusion of a nonsingular plane cubic $\Sigma$, and $L\in{\rm Pic}^0(\Sigma)$ is such that $h^0(\Sigma,L)=0$;
\item[(III)] $\overline{\calc(2,2,0)}$, whose generic point corresponds to instanton sheaves $E$ fitting into exact sequences of the form
\begin{equation}\label{C(2,2,0)}
0 \to E \to 2\cdot\op3\to \iota_*L(2) \to 0,
\end{equation}
where $\iota:\Sigma\hookrightarrow\p3$ is the inclusion of a nonsingular elliptic space quartic $\Sigma$, and $L\in{\rm Pic}^0(\Sigma)$ is such that $h^0(\Sigma,L)=0$;
\item[(IV)] $\overline{\calc(1,4,0)}$, whose generic point corresponds to instanton sheaves $E$ fitting into exact sequences of the form
\begin{equation}\label{C(1,4,0)}
0 \to E \to 2\cdot\op3\to \iota_*L(2) \to 0,
\end{equation}
where $\iota:\Sigma\hookrightarrow\p3$ is the inclusion of a nonsingular plane quartic $\Sigma$, and $L\in{\rm Pic}^2(\Sigma)$ is such that $h^0(\Sigma,L)=0$.
\end{itemize}

The first three components have dimension 29, and the last one has dimension 32; this is a classical result for the component $\cali^0(4)$, while the dimensions of $\calc(1,3,1)$,  $\calc(2,2,0)$ and $\calc(1,4,0)$ are given by Theorem \ref{dim C thm} above. Furthermore, \cite[Theorem 23]{JMT2} implies that each of the last three components intersects $\cali^0(4)$. Thus the union of these four components is connected.

To finish the proof of the second part of Main Theorem \ref{mthm1}, it is again enough to show that there are no other components in $\call(4)$, except for those described above. The argument here is the same as before, exploring Theorem \ref{jg-thm}, Remark \ref{remark 4} and Proposition \ref{components_4}.  

Take any $[E]\in\call(4)$ and consider the triple (\ref{std dual sqc}). Then, in view of Theorem \ref{jg-thm} and Remark \ref{remark 4}, $Q_E$ is a rank 0 instanton sheaf of multiplicity
$1\le d_E\le 4$, and $E^{\vee\vee}$ is an instanton bundle of charge $4-d_E$. Consider the possible cases for $d_E$.

\medskip

\noindent {$\mathbf{d_E=1.}$} As in the similar case in Section \ref{L(3) section},
$Q_E=\calo_{\ell}(1)$ where $l$ is a line in $\p3$. Respectively, $[E^{\vee\vee}]\in\cali(3)$. Deforming $\ell$ in $\p3$ we may assume that $E^{\vee\vee}|_{\ell}\simeq2\cdot\calo_{\ell}$, so that $[E]\in\cald(1,4)$. Therefore, $[E]\in\cali^0(4)$.

\medskip

\noindent{$\mathbf{d_E=2.}$} As in the similar case in Section \ref{L(3) section},
$Q_E$ can be deformed in a flat family either into a sheaf $\calo_\Gamma(3\mathrm{pt})$, where $\Gamma$ is a nonsingular conic in $\p3$, or into a sheaf $\calo_{\ell_1}(1)\oplus\calo_{\ell_2}(1)$ where $\ell_1$ and $\ell_2$ are skew lines in $\p3$. Respectively, $[E^{\vee\vee}]\in\cali(2)$. Now the same argument as in Section \ref{L(3) section} shows that $[E]\in\cali^0(4)$.

\medskip

\noindent{$\mathbf{d_E=3.}$} Then $E^{\vee\vee}$ is a null-correlation bundle
and, as in the case $d_E=3$ of Section \ref{L(3) section}, the sheaf $Q_E$ deforms in a flat family to one of the sheaves:

\begin{itemize}
\item[{\bf (i)}] $L(2)$, where $L\in\Pic^0(\Sigma)$, for some nonsingular plane cubic $\Sigma$ in $\p3$. 
\item[{\bf (ii)}] $\calo_\Delta(5\mathrm{pt})$, where $\Delta$ is a nonsingular twisted cubic in $\p3$;
\item[{\bf (iii)}] $\calo_\Gamma(3\mathrm{pt})\oplus\calo_{\ell}(1)$, where $\Gamma$ is a nonsingular conic and $\ell$ is a line disjoint from $\ell$;
\item[{\bf (iv)}] $\calo_{\ell_1}(1)\oplus\calo_{\ell_2}(1)\oplus
\calo_{\ell_3}(1)$ where $\ell_j$ are 3 skew lines in $\p3$.
\end{itemize}
By definition, $[E]\in\overline{\calc(1,3,1)}$ in case {\bf (i)}. The same argument as in Section \ref{L(3) section}, based on \cite[Theorem 7.8]{JMT1} and Proposition \ref{disjoint rat curves}, shows that $[E]\in\overline{\cali(4)}$ in cases {\bf (ii)} through {\bf (iv)}.

\medskip

\noindent$\mathbf{d_E=4.}$ Then $E^{\vee\vee}\simeq2\cdot\calo_{\p3}$ and,
according to Proposition \ref{components_4}, the sheaf $Q_E$ deforms in a flat family to one of the sheaves: 

\begin{itemize}
\item[{\bf (i)}] $L(2)$, where $L\in\Pic^2(\Sigma)$, for some nonsingular plane quartic $\Sigma$ in $\p3$, and $L$ satisfies an open condition $h^1(L)=0$; 
\item[{\bf (ii)}] $L(2)$, where $0\ne L\in\Pic^0(\Delta)$, for some nonsingular space elliptic quartic $\Delta$ in $\p3$; 
\item[{\bf (iii)}] $\calo_{\Delta}(7\mathrm{pt})$ for some nonsingular rational space quartic $\Delta$ in $\p3$; 
\item[{\bf (iv)}] $L(2)\oplus\calo_{\ell}(1)$, where $L\in\Pic^0(\Sigma)$, for some nonsingular plane cubic $\Sigma$ in $\p3$ and a line $\ell$ disjoint from $\Sigma$; 
\item[{\bf (v)}] $\calo_\Delta(5\mathrm{pt})\oplus\calo_{\ell}(1)$, where $\Delta$ is a nonsingular twisted cubic and $\ell$ is a line disjoint from $\Delta$;
\item[{\bf (vi)}] $\calo_\Gamma(3\mathrm{pt})\oplus\calo_{\ell_1}(1)
\oplus\calo_{\ell_2}(1)$, where $\Gamma$ is a nonsingular conic and $\ell_1,\ \ell_2$ are two skew lines disjoint from $\Gamma$;
\item[{\bf (vii)}] $\calo_{\ell}(1)\oplus\calo_{\ell_2}(1)
\oplus\calo_{\ell_3}(1)\oplus\calo_{\ell_4}(1)$, where $\ell_1,\ \ell_2,\ \ell_3,\ \ell_4$ are four disjoint lines in $\p3$. 
\end{itemize}

In case {\bf (i)}, since, in the notation of Lemma  $\ref{F,G}$, $F_0=2\cdot\calo_{\p3}$ and $G_0=L$, hence $H^i(\inhom(F_0,G_0))=0,\ \ \ i\ge1,$ and therefore the condition 
(\ref{vanish Hi}) is satisfied by the semicontinuity, so that the deformation argument as above shows that $E\in\overline{\calc(1,4,0)}$.

In case {\bf (ii)}, by the same reason, $[E]\in\overline{\call(2,2,0)}$. 

In case {\bf (iii)}, a similar argument shows that $[E]\in\overline{\cald(4,4)}$, and therefore $[E]\in\overline{\cali(4)}$.

In case {\bf (iv)}, Proposition \ref{disjoint curves} above guarantees that $[E]\in\overline{\call(1,3,1)}$.

In the cases remaining, {\bf (v)} through {\bf (vii)}, as in cases {\bf (ii)} and {\bf (iii)} for $d_E=3$ above, we again obtain $[E]\in\overline{\cali(4)}$.

Main Theorem \ref{mthm1} is finally proved.


\end{document}